\documentclass[twocolumn,english]{IEEEtran}
\usepackage[T1]{fontenc}
\usepackage{color}
\usepackage{babel}
\usepackage{textcomp}
\usepackage{amsthm}
\usepackage{amsmath}
\usepackage{amssymb}
\usepackage{graphicx}
\usepackage[unicode=true]
 {hyperref}
\usepackage{breakurl}

\makeatletter
  \theoremstyle{plain}
  \newtheorem{lem}{\protect\lemmaname}
  \theoremstyle{remark}
  \newtheorem{note}{\protect\notename}
  \theoremstyle{plain}
  \newtheorem{assumption}{\protect\assumptionname}
  \theoremstyle{plain}
  \newtheorem{prop}{\protect\propositionname}
 \ifx\proof\undefined\
   \newenvironment{proof}[1][\proofname]{\par
     \normalfont\topsep6\p@\@plus6\p@\relax
     \trivlist
     \itemindent\parindent
     \item[\hskip\labelsep
           \scshape
       #1]\ignorespaces
   }{%
     \endtrivlist\@endpefalse
   }
   \providecommand{\proofname}{Proof}
 \fi
  \theoremstyle{remark}
  \newtheorem{rem}{\protect\remarkname}
  \theoremstyle{plain}
  \newtheorem{thm}{\protect\theoremname}

\onecolumn

\makeatother

  \providecommand{\assumptionname}{Assumption}
  \providecommand{\notename}{Note}
\providecommand{\lemmaname}{Lemma}
\providecommand{\propositionname}{Proposition}
\providecommand{\remarkname}{Remark}
\providecommand{\theoremname}{Theorem}

\begin{document}

\title{\textbf{\Large{}Attitude Estimation and Control Using Linear-Like
Complementary Filters: Theory and Experiment}}

\author{L.~Benziane\textsuperscript{}, A.~El Hadri\textsuperscript{},
A.~Seba\textsuperscript{}, A.~Benallegue\textsuperscript{} and
Y. Chitour\textsuperscript{}%
\thanks{This research was partially supported by the iCODE institute, research
project of the Idex Paris-Saclay. L.~Benziane, A.~El Hadri, A.~Seba,
A. Benallegue are with LISV, Université de Versailles Saint Quentin,
France. Y. Chitour is with L2S, Université Paris-Sud XI, CNRS and
Supélec, Gif-sur-Yvette, and Team GECO, INRIA Saclay \textendash{}
Ile-de-France, France, e-mail: \protect\href{mailto:lotfi.benziane@ens.uvsq.fr,  Elhadri@lisv.uvsq.fr, Ali.seba@lisv.uvsq.fr, benalleg@lisv.uvsq.fr, yacine.chitour@lss.supelec.fr}{lotfi.benziane@ens.uvsq.fr,  Elhadri@lisv.uvsq.fr, Ali.seba@lisv.uvsq.fr, benalleg@lisv.uvsq.fr, yacine.chitour@lss.supelec.fr}.%
}}
\maketitle
\begin{abstract}
This paper proposes new algorithms for attitude estimation and control
based on fused inertial vector measurements using linear complementary
filters principle. First, $n$-order direct and passive complementary
filters combined with TRIAD algorithm are proposed to give attitude
estimation solutions. These solutions which are efficient with respect
to noise include the gyro bias estimation. Thereafter, the same principle
of data fusion is used to address the problem of attitude tracking
based on inertial vector measurements. Thus, instead of using noisy
raw measurements in the control law a new solution of control that
includes a linear-like complementary filter to deal with the noise
is proposed. The stability analysis of the tracking error dynamics
based on LaSalle's invariance theorem proved that almost all trajectories
converge asymptotically to the desired equilibrium. Experimental results,
obtained with DIY Quad equipped with the APM2.6 auto-pilot, show the
effectiveness and the performance of the proposed solutions. \end{abstract}


\section{Introduction}

Most of traditional rigid body attitude control approaches given in
the literature are based on feedback scheme using attitude estimation
\textcolor{black}{(see e.g.} \cite{Joshi1995,Thienel2003,Benallegue2008,Tayebi2008,Lee2013}).
Recently, some authors propose to use directly raw vector measurements
to perform attitude control \textcolor{black}{(see e.g. }\cite{Thakur2014b,Benziane2014,Tayebi2013}).
In fact, the explicit use of the attitude in the control law involves
the determination of attitude from measurements provided by appropriate
sensors. It is known that there are no sensors directly measuring
the attitude but it can be determined from measurements in the body
frame using suitable algorithms \textcolor{black}{(see e.g. }\cite{Wahba1965,Shuster1981,markley2000,Pflimlina2007,Batista2011}).
Almost all attitude control applications use measurement data from
embedded Inertial Measurement Units (IMU). The capability of the rigid
body to track desired attitude trajectories depends on the reliability
of these sensors and the quality of measurements related to sensitivity
to noise, bias, etc. To take into account the sensor imperfections,
many techniques of attitude estimation and control were developed.
For instance, as mentioned in the survey paper \cite{Crassidis2007},
the problem of attitude estimation is generally treated in two steps,
estimation of the attitude with raw measurements and filtering. The
most and widely used techniques in this case are based on extended
Kalman filter \cite{Crassidis2007,Jun1999}. Some other techniques
are developed like the nonlinear observer given in \cite{El2009},
or based on unscented filter \cite{Crassidis2003}. Most of these
methods are computationally demanding and some of them, depending
on used attitude representation \cite{Shuster1993}, suffer from topological
limitations, double covering or singularities. Another class of techniques
are based on complementary filters \cite{Euston2008,Vasconcelos2009}
which are not so computationally demanding, see \cite{Higgins1975}
for comparison between complementary and Kalman filtering.

Due to their simplicity and efficiency, the use of complementary filters
to reconstruct the attitude continues to attract many researchers.
A lot of them focus on low-cost IMU and attitude heading reference
system AHRS \cite{Martin2010}. Nonlinear complementary filters designed
on Special Orthogonal Group $SO(3)$ \cite{Mahony2008} and on the
unit 3-sphere $\mathbb{S}^{3}$ \cite{Tayebi2011} were used successfully
to estimate the attitude. Modified complementary filters using only
accelerometer and gyroscope measurements to estimate the orientation
was presented in \cite{Kubelka2012}. Another recent work has used
the inverse sensor models and complementary filters to develop a high-fidelity
attitude estimator \cite{Masuya2012}. As mentioned in \cite{Batista2012},
traditional attitude solutions use directly raw vector measurements
to compute the attitude data after that the observer is used to estimate
the attitude.\textit{ }\cite{Batista2012} proposed a reverse strategy
by combining a vector-based filter with an optimal attitude determination
algorithm, in which the distortion of noise characteristics is avoided.
A new interesting class of globally asymptotically stable filters
for attitude estimation was obtained. The vector-based filter was
designed as a Kalman filter using Linear Time Variant (LTV) representation
of the nonlinear kinematic equation. Even if experimental results
presented in \cite{Batista2012} are very good, the theoretical drawback
is the fact that the observability conclusions were given for the
LTV reformulation of the original nonlinear system and not explicitly
on the non linear system.

Inspired by approach given in \cite{Batista2012}, this paper presents
firstly globally asymptotically stable filters for attitude estimation
based on high order linear complementary filtering. The gyro-bias
estimation is also considered. Two forms of filter, termed \textit{``direct''}
and \textit{``passive''}, are designed similarly as the work presented
in \cite{Mahony2008}. The passive form is less sensitive to noise
as claimed in \cite{Mahony2008}. Moreover, the approach proposed
here is completely deterministic as it is based on linear complementary
filters followed by TRIAD algorithm for the attitude estimation. As
a matter of fact, the TRIAD is the deterministic attitude estimation
algorithm \textit{par excellence }as claimed by \cite{Shuster2006}.
Although it was proved that TRIAD is less accurate than other optimal
approaches \cite{Shuster2006}, we show throughout this work that
it is possible to obtain higher quality of the attitude estimation
when this approach is used.

The quality of IMU measurements is much degraded by the phenomenon
of vibrations of the real system. Frequently, the implementations
of some attitude controllers using directly raw vector measurements
are confronted with this phenomenon. Therefore, we propose to use
a new filter to improve the performance of the attitude tracking controller.
The proposed attitude controller is based on the filtered vector measurements
instead of the raw ones, while ensuring an almost global stability
without using ``attitude measurements''.

The result presented in this paper extends those from \cite{Benziane2012}.
The first contribution of this work is the extension of the global
convergence of the direct complementary filters to the case of n-order.
Also, we propose general n-order passive filters, where we obtain
the global asymptotic convergence to zero of the estimation errors.
This constitutes our second contribution. Another contribution is
the design of a new control law based only on inertial and rate-gyro
measurements to control the attitude of a rigid body without using
``attitude measurements\textquotedblright , for which an almost global
stability is given. All our contributions are validated by experiments
on the DIY drone Quad-copter \cite{3DRa}.

\section{\label{sec:Preliminaries}Preliminaries}

\subsection{Mathematical background and Notations }

Consider a rigid-body moving in 3D space with orthonormal body-frame
$\left\{ \mathcal{B}\right\} $ fixed to its center of gravity and
denote by $\left\{ \mathcal{I}\right\} $ the inertial reference frame
attached to the 3D space\textcolor{black}{. Attitude of the rigid
body represents the relative orientation of the} $\left\{ \mathcal{B}\right\} $\textcolor{black}{{}
with respect to }$\left\{ \mathcal{I}\right\} $\textcolor{black}{.
It }can be represented using several mathematical models\textcolor{black}{.
Representing the attitude by rotation matrix $R$, provides an unique,
global and }non singular\textcolor{black}{{} }parametrization\textcolor{black}{{}
of the orientation \cite{Shuster1993}. The rotation matrix is an
element of the special orthogonal group }$SO(3)$ with\textcolor{black}{{}
}$SO(3)=\{R\in\mathcal{\mathrm{\mathbb{R}}}^{3\times3}\mid R^{T}R=RR^{T}=I_{3},\,\det(R)=1\}$
where $I_{3}$ is the $3\times3$ identity matrix. The associated
Lie algebra denoted by $\mathfrak{so}(3)$ is the set of skew symmetric
matrices such that $\mathfrak{so}(3)=\{A\in\mathcal{\mathrm{\mathbb{R}}}^{3\times3}\mid A=-A^{T}\}$.
Denote by $S$ the Lie algebra mapping from $\mathbb{R}^{3}\rightarrow\mathfrak{so}(3)$
which associates to $x\in\mathbb{R}^{3}$ the skew-symmetric matrix
$S(x)$, such that

\begin{equation}
S(x)=\left[\begin{array}{ccc}
0 & -x_{z} & x_{y}\\
x_{z} & 0 & -x_{x}\\
-x_{y} & x_{x} & 0
\end{array}\right]\; and\; x=\left[\begin{array}{c}
x_{x}\\
x_{y}\\
x_{z}
\end{array}\right]\label{eq:definition of skew sym matrix}
\end{equation}

For any two vectors $x,\, y\in\mathbb{R}^{3}$ and rotation matrix
$R\in SO(3)$, the following identities hold: 
\begin{equation}
\begin{cases}
S(x)y & =x\times y=-S(y)x,\\
S(S(x)y) & =S(x)S(y)-S(y)S(x),\\
S(x)^{2} & =xx^{T}-x^{T}xI_{3},\\
S(Rx) & =RS(x)R^{T},
\end{cases}\label{S(x)y}
\end{equation}
where \textcolor{black}{$\times$ denotes the vector cross product.}

Another global and non singular parametrization\textcolor{black}{{}
of the attitude is described by} unit quaternion $Q$ which is an
element of unit sphere $\mathbb{S}^{3}=\left\{ Q=\left(q_{0},q^{T}\right)^{T}\text{, }q_{0}\in\mathbb{R},\text{\ }q\in\mathbb{R}^{3},\text{ }q_{0}^{2}+q^{T}q=1\right\} $.
The multiplication of two quaternions $P=(p_{0},p^{T})^{T}$ and $Q=(q_{0},q^{T})^{T}$
is denoted by ``$\odot$'' and defined as $P\odot Q=\left[\begin{array}{c}
p_{0}q_{0}-p^{T}q\\
p_{0}q+q_{0}p+p\times q
\end{array}\right]$ and for any unit quaternion $Q=(q_{0},q^{T})^{T}$, we have $Q\odot Q^{-1}=Q^{-1}\odot Q=(1,\mathbf{0})$,
where $Q^{-1}=(q_{0},-q^{T})^{T}$.

Both $Q\in\mathbb{S}^{3}$ and $R\in SO(3)$ are related to each other
through the mapping $\mathcal{R}:\mathbb{S}^{3}\rightarrow SO(3)$
by the Euler-Rodriguez formula as follows: 
\begin{equation}
\begin{array}{c}
\mathcal{R}(Q)=I_{3}+2q_{0}S(q)+2S(q)^{2}\end{array}\label{eq: rodriguez}
\end{equation}

If $n$ is a positive integer, set $e_{n}=(0,\cdots,0,1)^{T}$. To
every $\gamma=(\gamma_{1},\ldots,\gamma_{n})\in\mathbb{R}^{n}$, we
associate the polynomial 
\begin{equation}
P_{\gamma}(s)=s^{n}+\sum_{k=1}^{n}\gamma_{k}s^{n-k},\label{eq:definition of caracteristic polynomial N order}
\end{equation}
and the companion matrix $A_{\gamma}$ 
\begin{equation}
A_{\gamma}=\left(\begin{array}{ccccc}
0 & 1 & 0 & \cdots & 0\\
0 & 0 & 1 & \ddots & \vdots\\
\vdots & \vdots & 0 & \ddots & \vdots\\
\vdots & \vdots & \vdots & \ddots & 0\\
0 & 0 & \cdots & 0 & 1\\
-\gamma_{n} & -\gamma_{n-1} & \cdots & -\gamma_{2} & -\gamma_{1}
\end{array}\right)\label{eq:definition of campagnon matrix n order}
\end{equation}
whose characteristic polynomial is $P_{\gamma}$. Use $\pi:\mathbb{R}^{n}\rightarrow\mathbb{R}^{n-1}$
to denote the projection onto $\mathbb{R}^{n-1}$ i.e., $\pi(\gamma)=(\gamma_{1},\cdots,\gamma_{n-1})$.
Define the following subsets of $\mathbb{R}^{n}$, 
\[
\mathcal{H}_{n}=\left\{ \gamma\in\mathbb{R}^{n}\,\mid\, P_{\gamma}\; Hurwitz\right\} ,\quad\mathcal{\overline{H}}_{n}=\left\{ \gamma\in\mathcal{H}_{n}\,\mid\,\pi(\gamma)\in\mathcal{H}_{n-1}\right\} .
\]
The proof of the following lemma is defereed in Appendix. 
\begin{lem}
\label{note: Hn_bar is not empty}If $n$ is a positive integr, then
$\mathcal{\overline{H}}_{n}$ is not empty. \end{lem}
\begin{note}
\label{note: Hurwitz kronecker product matrix} Let $E\in\mathbb{R}^{(n\times n)}$
and $\sigma(E)=\left\{ \lambda_{1},\ldots,\lambda_{n}\right\} $ its
spectrum, where $\lambda_{l},\, l=1\ldots n$ are the eigenvalues
of $E$. Let $I_{k}\in\mathbb{R}^{(k\times k)}$, $k$ integer, be
the identity matrix. Then, the spectrum of the Kronecker product of
$E$ by $I_{k}$, $E\otimes I_{k}\in\mathbb{R}^{(kn\times kn)}$,
is equal to $\sigma(E)$ according to Theorem in page 245 of \cite{Horn1991}.
In particular, $E\otimes I_{k}$ is Hurwitz if and only $E$ is. 
\end{note}

\subsection{\label{sub:Attitude-kinematics,-Dynamics, assumptions}Attitude kinematics,
Dynamics and Assumptions}

\textcolor{black}{The rigid body rotational motion can be described
by its kinematic and dynamic equations. Using the rotation matrix
representation, the rigid body }attitude is governed by the following
kinematic equation 
\begin{equation}
\dot{R}(t)=R(t)S(\omega(t)),\label{eq: attitude kinematics R}
\end{equation}
\textcolor{black}{where }$\omega(t)$\textcolor{black}{{} being the
angular velocity of the rigid body expressed in }$\left\{ \mathcal{B}\right\} $.
Equivalently in term of unit quaternion, we can have $\dot{Q}(t)=\frac{1}{2}Q(t)\odot\bar{\omega}(t)$
with $\bar{\omega}(t)$ is the pure quaternion defined by $\bar{\omega}(t)=(0,\omega(t)^{T})^{T}$,
which gives

\begin{equation}
\dot{Q}(t)=\left[\begin{array}{c}
\dot{q}_{0}(t)\\
\dot{q}(t)
\end{array}\right]=\left[\begin{array}{c}
-\frac{1}{2}q^{T}(t)\omega(t)\\
\frac{1}{2}(q_{0}(t)I_{d}+S(q(t)))\omega(t)
\end{array}\right],\label{eq:real kinematics}
\end{equation}
Now, given a constant vector $r$ in inertial $\left\{ \mathcal{I}\right\} $,
then its \textcolor{black}{corresponding} vector in the $\left\{ \mathcal{B}\right\} $
is given by $b(t)=R^{T}(t)r$. Thus, using (\ref{eq: attitude kinematics R}),
one can get the following reduced attitude kinematics

\begin{equation}
\dot{b}(t)=-S(\omega(t))b(t)\label{dynamic_bi}
\end{equation}
By considering applied torque $\tau(t)$ to the system expressed in
$\left\{ \mathcal{B}\right\} $, the rigid body simplified rotational
dynamics is governed by 
\begin{equation}
\begin{array}{c}
J\dot{\omega}(t)=-S(\omega(t))J\omega(t)+\tau(t),\end{array}\label{eq. dynamic de rotation}
\end{equation}
where $J\in\mathbb{R}^{3\times3}$ is a symmetric positive definite
constant inertia matrix of the rigid body with respect to $\left\{ \mathcal{B}\right\} $.

Consider the following rate-gyros model 
\begin{equation}
\omega_{m}(t)=\omega(t)+\eta,\label{eq:Rate gyro model}
\end{equation}
where $\omega_{m}(t)$ is the measured angular velocity and $\eta$
is the real unknown gyro-bias.

Along this work, we use the following assumptions : 
\begin{assumption}
\label{assump:Colinearity of bi}We assume that, if we have $m$ measured
vectors $b_{i}(t),\, i=1,...,m$ expressed in $\left\{ \mathcal{B}\right\} $,
corresponding to $m$ inertial constant vectors $r_{i},\, i=1,...,m$
expressed in $\left\{ \mathcal{I}\right\} $, then at least two of
them are non-collinear. 
\end{assumption}

\begin{assumption}
\label{assump:boundness omega omega dot} We assume that the real
unknown gyro-bias $\eta$ is bounded and constant (or slowly varying),
such that $\dot{\eta}=0$. Moreover, we assume that we are dealing
with bounded measured angular velocities $\omega_{m}(\cdot)$, implying
that the real angular velocity $\omega(\cdot)$ is bounded as well. 
\end{assumption}
Using the reduced attitude kinematics (\ref{dynamic_bi}) and the
model of the rate-gyro (\ref{eq:Rate gyro model}), we can write the
following system 
\begin{equation}
\begin{cases}
\begin{array}{lcl}
\dot{b}_{i} & = & -S(\omega_{m}-\eta)b_{i}\\
\dot{\eta} & = & 0.
\end{array}\end{cases},\label{eq: real bi dynamics}
\end{equation}
where $i=1,...,m$. Note that, in all what follows the indices $i=1,...,m$
denote the number of the used inertial vectors.

\subsection{Complementary linear filter-based attitude estimation approach}

The sensor-based attitude estimation approach \cite{Batista2012}
is consisting of two processes: i) filtering sensor measurements,
and ii) determining attitude. Inspired by this approach, we propose
a structure based on complementary linear filter rather than sensor-based
filter method. Indeed, complementary filters give us a mean to fuse
multiple heterogeneous independent noisy measurements of the same
signal that have complementary spectral characteristics \cite{Mahony2008}.
By developing a high-fidelity and simple algorithm for attitude estimation,
the proposed structure must allow the possibility of using high order
filter which leads to better performance.

Using the reduced attitude kinematics (\ref{dynamic_bi}), the complementary
filter model for fusing the measured inertial vector $b_{i}(t)$ and
gyros measurements $\omega_{m}$ in order to get estimate $\hat{b}_{i}(t)$
is shown in Figure \ref{fig:Complementary classical form}, where
the notion of complementary filter is achieved if the following condition
is satisfied 
\begin{equation}
H_{1i}(s)+sH_{2i}(s)=1,\quad i=1,\cdots,m,\label{com_filter_eq}
\end{equation}
where $H_{1i}(s)$ is a low-pass filter and $sH_{2i}(s)$ is a high-pass
filter.

\begin{center}
\begin{figure}[th]
\centering{}\includegraphics[width=9cm]{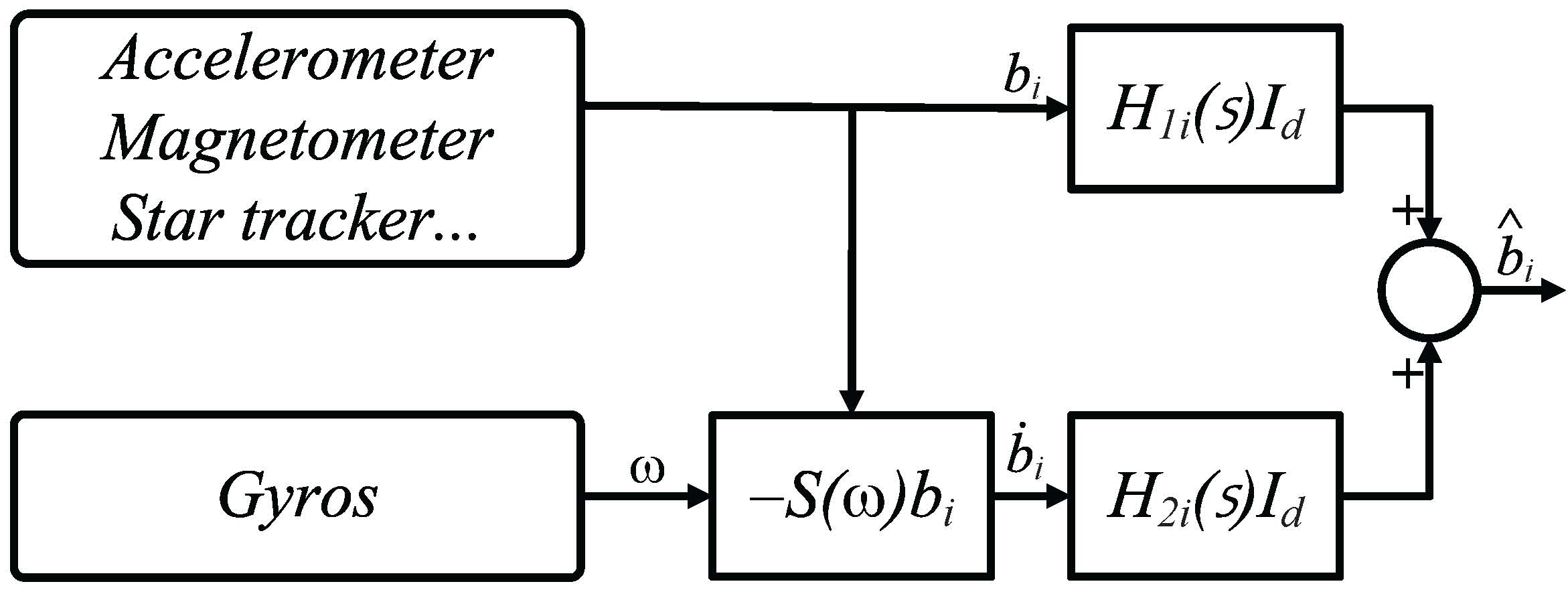}\protect\caption{\label{fig:Complementary classical form}Classical form of complementary
filter}
\end{figure}

\par\end{center}

From the structure of the complementary filter given in Figure \ref{fig:Complementary classical form},
the estimate $\hat{b}_{i}$ of the state $b_{i}$ by fusing measurements
of $i\, th$ inertial direction vector and gyro measurements can be
write as 
\begin{equation}
\hat{b}_{i}=H_{1i}(s)b_{i}+H_{2i}(s)\dot{b}_{i},\quad i=1,\cdots,m.\label{eq:b_hat_fromCF}
\end{equation}

Now, for the determination of the attitude, the complementary filter
can will be followed by a TRIAD algorithm \cite{Shuster1981}. Despite
the fact that TRIAD is known less accurate than other statistical
algorithms based on minimizing Wahba's loss function \cite{Shuster2006},
we will show that we can obtain good results by using fused data.
The choice of TRIAD algorithm is justified by the fact that optimal
algorithms are usually much slower than deterministic algorithms \cite{Shuster1981,Shuster2006}.

The first problem addressed in this work is the design of an attitude
and heading reference system using the concept of sensor-based attitude
estimation approach \cite{Batista2012}. The goal is to proof that
it is possible to obtain a structure based on complementary linear
filter with a globally asymptotic convergence. The filtered data will
be used by a TRIAD for attitude determination as explained before.

The second problem addressed is to proof that the use of estimated
measurements by complementary filters can achieve attitude tracking
with an almost global stability.

\section{\label{sec:Attitude-and-Gyro-Bias}Design of High Order Direct and
Passive Filters with Gyro-Bias Estimation}

The principle of the \textit{``classical form''} of complementary
filters is based on the data fusion of measurements of inertial direction
vectors and gyro measurements as depicted by the scheme of Figure
\ref{fig:Complementary classical form}. This scheme can be reformulated
in \textit{``feedback form''} as shown by Figure \ref{fig:Direct-form-of}.
Furthermore, according to the manner of offsetting the nonlinear term,
we can obtain two structures of the complementary filter. The first
one is termed \textit{``direct linear complementary filter''} and
the second one termed \textit{``passive linear-like complementary
filter''} . Indeed, in the first one, the offsetting of nonlinear
term uses direct raw measurements as shown in Figure \ref{fig:Direct-form-of}
while in the second one, the filtered measurements are used as depicted
in Figure \ref{fig:Passive-linear-complementary}.

\begin{center}
\begin{figure}[t]
\centering{}\includegraphics[width=9cm]{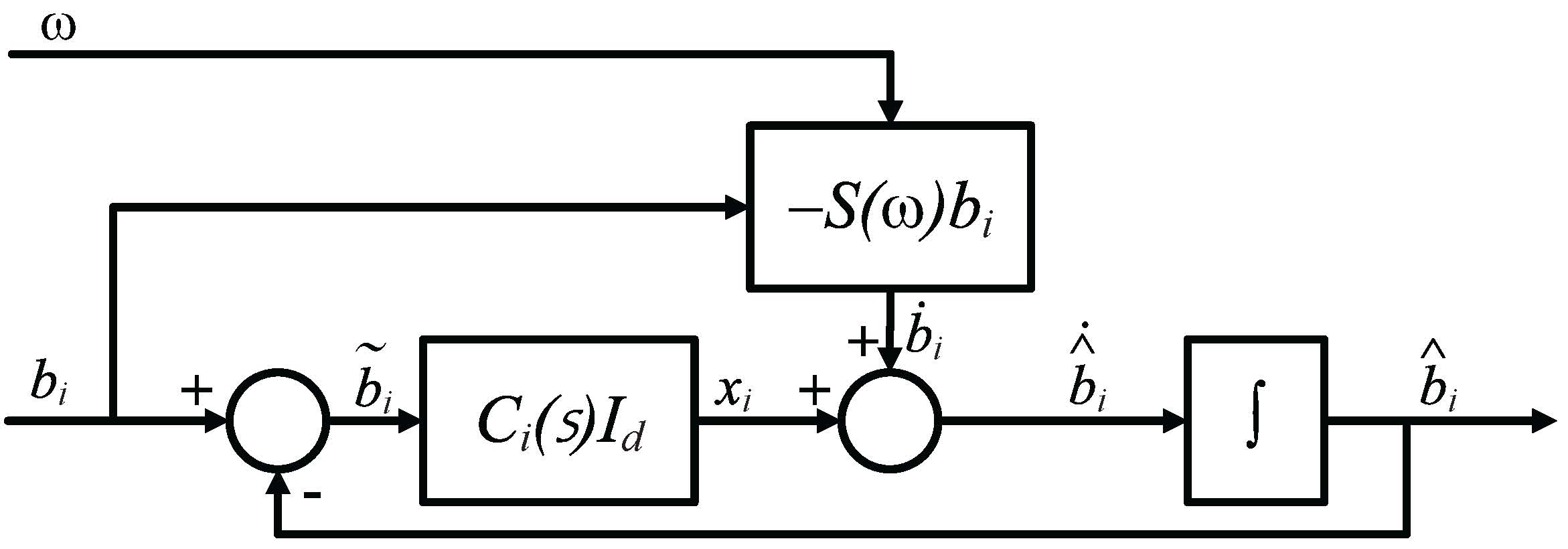}\protect\caption{\label{fig:Direct-form-of}Direct linear complementary filter}
\end{figure}

\par\end{center}

\noindent \begin{center}
\begin{figure}[th]
\centering{}\includegraphics[width=9cm]{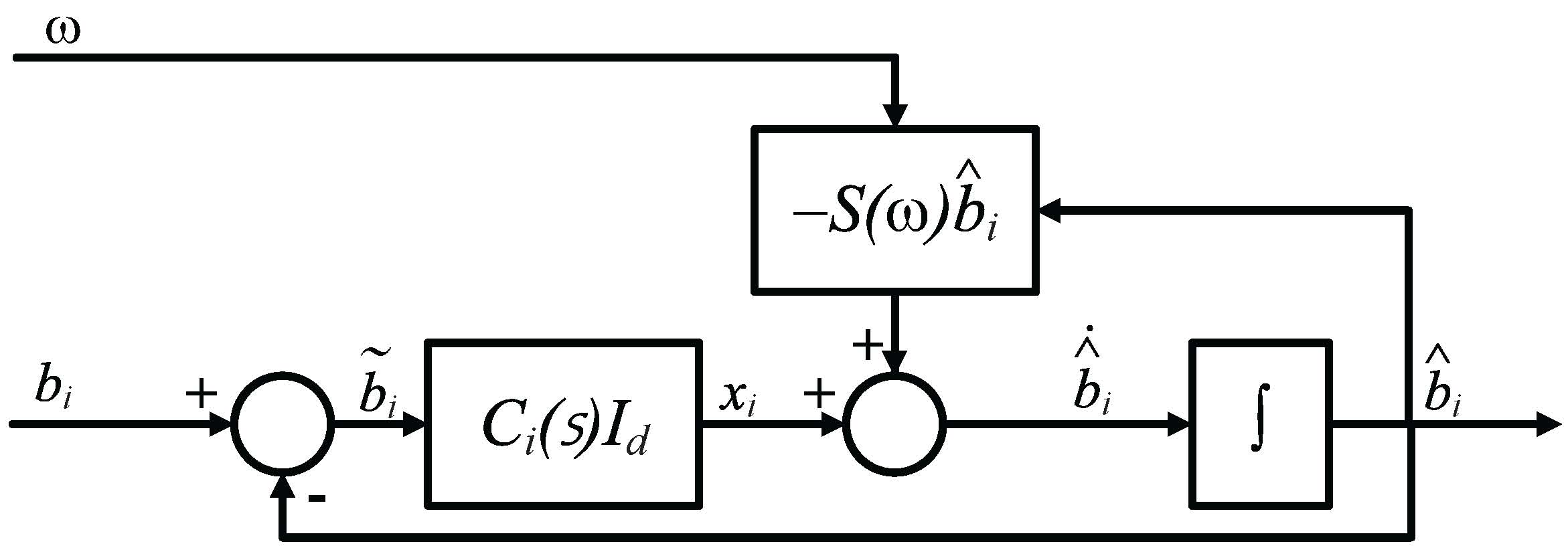}\protect\caption{\label{fig:Passive-linear-complementary}Passive linear-like complementary
filter}
\end{figure}

\par\end{center}

From the equivalence between the \textit{``classical form''} and
the \textit{``feedback form''}, one can get 
\begin{equation}
H_{1i}\left(s\right)=\frac{C_{i}\left(s\right)}{s+C_{i}\left(s\right)},\; H_{2i}\left(s\right)=\frac{1}{s+C_{i}\left(s\right)},\, i=1,...,m,\label{tf_H1_H2}
\end{equation}
where $C_{i}(s)$ represents the compensator term in the feedback
form. From (\ref{tf_H1_H2}), we can write the compensator term as

\begin{equation}
C_{i}(s)=\frac{sH_{1i}(s)}{1-H_{1i}(s)},\, i=1,...,m.\label{direct_filter_C}
\end{equation}
The design of the compensator $C_{i}(s)$ can be achieved by choosing
the adequate filter order for improving the quality of estimation.
Consider now, for $i=1\cdots,m,$ the general $n$-order transfer
function $H_{1i}(s)$ by first taking $\varUpsilon_{i}\in\mathcal{H}_{n}$
and setting 
\begin{equation}
H_{1i}(s)=\frac{\gamma_{in}}{P_{\varUpsilon_{i}}(s)},\label{eq:definition of H1 in n-order form}
\end{equation}
where $P_{\varUpsilon_{i}}(s)$ and $\gamma_{in}$ are defined by
(\ref{eq:definition of caracteristic polynomial N order}). Using
(\ref{direct_filter_C}) one can get 
\begin{equation}
C_{i}(s)=\frac{\gamma_{in}}{P_{\varPi(\varUpsilon_{i})}(s)},\, i=1,...,m.\label{eq: transfert_func_direct_n}
\end{equation}

\subsection{High-Order Direct Linear Complementary Filters}

Consider System (\ref{eq: real bi dynamics}) and the block diagram
of the direct form in Figure \ref{fig:Direct-form-of} with compensator
$C_{i}(s)$ given by (\ref{eq: transfert_func_direct_n}) for $i=1,...,m$.
Then, the closed-loop dynamics with gyro bias estimation for any $n$-order
is given for $i=1,...,m$ by

\begin{equation}
\begin{cases}
\begin{array}{lcl}
x_{i}^{(n-1)} & = & -\sum_{k=1}^{n-1}\gamma_{ik}x_{i}^{(n-k-1)}+\gamma_{in}(b_{i}-\hat{b}_{i}),\\
\dot{\hat{b}}_{i} & = & -S(\omega_{m}-\hat{\eta})b_{i}+x_{i},\\
\dot{\hat{\eta}} & = & \Gamma_{d}\sum_{i=1}^{m}S(b_{i})\upsilon_{i}
\end{array}\end{cases},\label{Eq. direct n ordre case}
\end{equation}
where $x_{i}^{(j)}$ is the $j-$th derivative of $x_{i}$ with $x_{i}^{(0)}=x_{i}$,
$\gamma_{ik},\, i=1,...,m,\, k=1,...,n$ are components of $\varUpsilon_{i}\in\mathcal{H}_{n}$,
$\Gamma_{d}$ is a real positive definite diagonal matrix gain and
$\upsilon_{i}$ is a vector to be defined later.

Define the observation errors 
\begin{eqnarray}
\widetilde{b}_{i} & = & b_{i}-\hat{b}_{i},\, i=1,...,m,\label{eq: bi tild error definition}\\
\widetilde{\eta} & = & \eta-\hat{\eta},\label{eq:eta tild error definition}
\end{eqnarray}
then using (\ref{dynamic_bi}) and (\ref{Eq. direct n ordre case})-(\ref{eq:eta tild error definition}),
yield the following error dynamics

\begin{equation}
\begin{cases}
\begin{array}{lcl}
x_{i}^{(n-1)} & = & -\sum_{k=1}^{n-1}\gamma_{ik}x_{i}^{(n-k-1)}+\gamma_{in}\tilde{b}_{i},\\
\dot{\widetilde{b}}_{i} & = & -S(b_{i})\widetilde{\eta}-x_{i},\\
\dot{\widetilde{\eta}} & = & -\Gamma_{d}\sum_{i=1}^{m}S(b_{i})\upsilon_{i}.
\end{array}\end{cases}\label{Eq. err direct n order model}
\end{equation}
By the evaluation of the time derivative of the first equation of
(\ref{Eq. err direct n order model}), one can rewrite (\ref{Eq. err direct n order model})
as

\begin{equation}
\begin{cases}
\begin{array}{lcl}
x_{i}^{(n)} & = & -\sum_{k=1}^{n}\gamma_{ik}x_{i}^{(n-k)}-\gamma_{in}S(b_{i})\widetilde{\eta}\\
\dot{\widetilde{\eta}} & = & -\Gamma_{d}\sum_{i=1}^{m}S(b_{i})\upsilon_{i}
\end{array}\end{cases}\label{Eq. err direct n order model-1}
\end{equation}

Now, consider the new state vector $z_{i}\in\mathbb{R}^{3n},\, i=1,...,m$
such as $z_{i}^{T}=\left[x^{T},\dot{x}^{T},\cdots,x^{(n-1)T}\right]$
and define the vectors $\upsilon_{i}$ to be 
\begin{equation}
\upsilon_{i}=B_{di}^{T}P_{di}z_{i},\, i=1,...,m.\label{eq: definition du vecteur vi}
\end{equation}
One can rewrite (\ref{Eq. err direct n order model-1}) as 
\begin{equation}
\begin{cases}
\begin{array}{lcl}
\dot{z}_{i}(t) & = & A_{di}z_{i}(t)+B_{di}S(\widetilde{\eta})b_{i},\\
\dot{\widetilde{\eta}} & = & -\Gamma_{d}\sum_{i=1}^{m}S(b_{i})B_{di}^{T}P_{di}z_{i},
\end{array}\end{cases}\label{Eq. err direct n order model-1-1}
\end{equation}
where $i=1,...,m$, the Hurwitz matrices $A_{di}=A_{\varUpsilon_{i}}\otimes I_{d}\in\mathbb{R}^{(3n\times3n)}$
($A_{\varUpsilon_{i}}$ is defined by (\ref{eq:definition of campagnon matrix n order})),
$B_{di}=\gamma_{in}e_{n}\otimes I_{3}\in\mathbb{R}^{3n\times3}$ and
the matrices $P_{di}\in\mathbb{R}^{(3n\times3n)},\, i=1,...,m$, are
real symmetric positive definite solutions of the following Lyapunov
equations for given symmetric positive definite matrices $Q_{di}$
\begin{equation}
A_{di}^{T}P_{di}+P_{di}A_{di}=-Q_{di},\, i=1,...,m\label{eq:Lyapunov equation Pdi}
\end{equation}

We can now state our first result. 
\begin{prop}
\label{prop: High order Direct proposition}Consider the filter (\ref{Eq. direct n ordre case})
and (\ref{eq: definition du vecteur vi}), under Assumptions \ref{assump:Colinearity of bi}
and \ref{assump:boundness omega omega dot} in subsection \ref{sub:Attitude-kinematics,-Dynamics, assumptions}.
Then the errors (\ref{eq: bi tild error definition}) and (\ref{eq:eta tild error definition})
converge globally asymptotically to zero.\end{prop}
\begin{proof}
Consider the following Lyapunov function candidate 
\begin{equation}
V_{1}=\sum_{i=1}^{m}z_{i}^{T}P_{di}z_{i}+\widetilde{\eta}^{T}\Gamma_{d}^{-1}\widetilde{\eta}\label{Lyap1}
\end{equation}
where $P_{di}\in\mathbb{R}^{(3n\times3n)},\, i=1,...,m$ is given
by (\ref{eq:Lyapunov equation Pdi}). The time derivative of (\ref{Lyap1})
in view of (\ref{Eq. err direct n order model-1-1}) is given by 
\[
\begin{array}[t]{ccl}
\dot{V}_{1} & = & \sum_{i=1}^{m}\left(z_{i}^{T}P_{di}\dot{z}_{i}+\dot{z}_{i}^{T}P_{di}z_{i}\right)+\widetilde{\eta}^{T}\Gamma_{d}^{-1}\dot{\widetilde{\eta}},\\
 & = & \sum_{i=1}^{m}\left(z_{i}^{T}\left(A_{di}^{T}P_{di}+P_{di}A_{di}\right)z_{i}+2z_{i}^{T}P_{di}B_{di}S(\widetilde{\eta})b_{i}\right)-2\widetilde{\eta}^{T}\sum_{i=1}^{m}S(b_{i})B_{di}^{T}P_{di}z_{i},
\end{array}
\]
using (\ref{eq:Lyapunov equation Pdi}) and the fact that $\widetilde{\eta}^{T}S(b_{i})B_{di}^{T}P_{di}z_{i}=z_{i}^{T}P_{di}B_{di}S(\widetilde{\eta})b_{i}$,
then 
\begin{equation}
\dot{V}_{1}=-\sum_{i=1}^{m}z_{i}^{T}Q_{di}z_{i}\leqslant0.\label{eq:New Lyap func}
\end{equation}
Therefore $z_{i}$ and $\widetilde{\eta}_{i}$ are bounded and consequently
by using (\ref{Eq. err direct n order model-1-1}), $\dot{z}_{i}$
and $\dot{\widetilde{\eta}}_{i}$ are bounded. The evaluation of the
second derivative of (\ref{Lyap1}) in view of (\ref{Eq. err direct n order model-1-1})
gives 
\begin{eqnarray}
\ddot{V}_{1} & = & -\sum_{i=1}^{m}z_{i}^{T}\left(A_{di}^{T}Q_{di}+Q_{di}A_{di}\right)z_{i}+2z_{i}^{T}Q_{di}B_{di}S(b_{i})\widetilde{\eta},\label{dt_dt_Lyap}
\end{eqnarray}
which is clearly bounded. By Barbalat's lemma, ${\displaystyle \lim_{t\rightarrow\infty}\dot{V}_{1}(t)=0}$
and consequently ${\displaystyle \lim_{t\rightarrow\infty}z_{i}(t)=0}$.
Then, according to (\ref{Eq. err direct n order model}), one can
obtain ${\displaystyle \lim_{t\rightarrow\infty}\widetilde{b}_{i}(t)=0}$.
The second time derivative of $z_{i}$ is given by 
\begin{equation}
\ddot{z}_{i}=A_{di}\left(A_{di}z_{di}(t)+B_{di}S(b_{i})\widetilde{\eta}\right)+B_{di}S(S(b_{i})\omega)\widetilde{\eta}+B_{di}S(b_{i})\dot{\widetilde{\eta}},\label{bpp}
\end{equation}
where all terms are bounded. Thus using Barbalat's lemma, ${\displaystyle \lim_{t\rightarrow\infty}\dot{z}_{i}(t)=0}$.
Therefore, using (\ref{Eq. err direct n order model-1-1}) and ${\displaystyle \lim_{t\rightarrow\infty}z_{i}(t)=0}$,
one can conclude that $B_{di}S(b_{i})\widetilde{\eta}$ converge to
zero and equivalently ${\displaystyle \lim_{t\rightarrow\infty}S(b_{i}(t))\widetilde{\eta}(t)=0}$.
Under Assumption \ref{assump:Colinearity of bi}, one can conclude
that ${\displaystyle \lim_{t\rightarrow\infty}\widetilde{\eta}(t)=0}$. \end{proof}
\begin{rem}
Substituting the value of $n$ by 1 in (\ref{Eq. direct n ordre case}),
and after some manipulations, one can obtain the first order direct
filter as 
\begin{equation}
\begin{cases}
\begin{array}{lcl}
\dot{\hat{b}}_{i} & = & -S(\omega_{m}-\hat{\eta})b_{i}+\gamma_{i1}(b_{i}-\hat{b}_{i})\\
\dot{\hat{\eta}} & = & \Gamma_{1}\sum_{i=1}^{m}S(b_{i})\hat{b}_{i}
\end{array}\end{cases},\label{Eq. direct n ordre case-1}
\end{equation}

\end{rem}

\subsection{High-Order Passive Linear-like Filters}

In the passive form, the design of the complementary filter is performed
by injecting filtered measurements for offsetting nonlinear term as
shown in block diagram of Figure \ref{fig:Passive-linear-complementary}
with a compensator $C_{i}(s)$, $i=1,\cdots,m$, defined by (\ref{eq: transfert_func_direct_n}).
Then, we propose the following new $n$-order passive form with gyro
bias estimation

\begin{equation}
\begin{cases}
\begin{array}{lcl}
x_{i}^{(n-1)} & = & -\sum_{k=1}^{n-1}\gamma_{ik}x_{i}^{(n-k-1)}+\gamma_{in}(b_{i}-\hat{b}_{i}),\\
\dot{\hat{b}}_{i} & = & -S(\omega_{m}-\hat{\eta})\hat{b}_{i}+w_{i},\\
\dot{\hat{\eta}} & = & -\Gamma_{p}\sum_{i=1}^{m}S(b_{i})\hat{b}_{i},
\end{array}\end{cases},\label{Eq. passive n ordre case}
\end{equation}
where $i=1,...,m$, $x_{i}^{(j)}$ is the $j\, th$ order derivative
of $x_{i}$ with $x_{i}^{(0)}=x_{i}$, $\gamma_{ik},\, i=1,...,m,\, k=1,...,(n-1)$
are components of $\pi(\varUpsilon_{i})$ for $\varUpsilon_{i}\in\mathcal{\overline{H}}_{n}$,
$\Gamma_{p}$ is a real positive definite diagonal matrix gain and
$w_{i}$ are given by 
\begin{equation}
w_{i}=B_{pi}^{T}P_{pi}X_{i},\label{eq:definition of w_i}
\end{equation}
with $X_{i}\in\mathbb{R}^{3(n-1)},\, i=1,...,m$ such as $X_{i}^{T}=\left[x^{T},\dot{x}^{T},\cdots,x^{(n-2)T}\right]$,
allowing to rewrite (\ref{Eq. passive n ordre case}) as

\begin{equation}
\begin{cases}
\begin{array}{lcl}
\dot{X}_{i}(t) & = & A_{pi}X_{i}(t)+B_{pi}(b_{i}-\hat{b}_{i}),\\
\dot{\hat{b}}_{i} & = & -S(\omega_{m}-\hat{\eta})\hat{b}_{i}+B_{pi}^{T}P_{pi}X_{i},\\
\dot{\hat{\eta}} & = & -\Gamma_{p}\sum_{i=1}^{m}S(b_{i})\hat{b}_{i},
\end{array}\end{cases}\label{Eq. Vectoriel definition of passive n order model}
\end{equation}
where the Hurwitz matrices $A_{pi}=A_{\varPi(\varUpsilon_{i})}\otimes I_{d}\in\mathbb{R}^{(3(n-1)\times3(n-1))}$
($A_{\varPi(\varUpsilon_{i})}$ is defined by (\ref{eq:definition of campagnon matrix n order})),
see Note \ref{note: Hurwitz kronecker product matrix} for $A_{pi}$
Hurwitz) and the matrices $B_{pi}=\gamma_{in}e_{(n-1)}\otimes I_{d}\in\mathbb{R}^{3(n-1)\times3}$
and the matrices $P_{pi}\in\mathbb{R}^{(3(n-1)\times3(n-1))},\, i=1,...,m$,
are real symmetric positive definite solutions of the following Lyapunov
equations for given symmetric positive definite matrices $Q_{pi}$
\begin{equation}
A_{pi}^{T}P_{pi}+P_{pi}A_{pi}=-Q_{pi},\label{eq:Lyapunov equation Ppi}
\end{equation}

We now state our second result. 
\begin{prop}
\label{proposition: High order passive filter} Consider the filter
(\ref{Eq. passive n ordre case}), under Assumptions \ref{assump:Colinearity of bi}
and \ref{assump:boundness omega omega dot} in subsection \ref{sub:Attitude-kinematics,-Dynamics, assumptions}.
Then the errors (\ref{eq: bi tild error definition}) and (\ref{eq:eta tild error definition})
converge globally asymptotically to zero.\end{prop}
\begin{proof}
First let us evaluate the error dynamics of (\ref{Eq. Vectoriel definition of passive n order model}).
Using (\ref{dynamic_bi}) and (\ref{eq: bi tild error definition}),(\ref{eq:eta tild error definition}),
one can get 
\begin{equation}
\begin{cases}
\begin{array}{lcl}
\dot{X}_{i}(t) & = & A_{pi}X_{i}(t)+B_{pi}\tilde{b}_{i},\\
\dot{\widetilde{b}}_{i} & = & -S(b_{i})\widetilde{\eta}+S(\tilde{b}_{i})(\omega+\widetilde{\eta})-B_{pi}^{T}P_{pi}X_{i},\\
\dot{\widetilde{\eta}} & = & -\Gamma_{p}\sum_{i=1}^{m}S(b_{i})\tilde{b}_{i},
\end{array}\end{cases}\label{Eq. err passive n order model-2-1}
\end{equation}

Consider now, the following Lyapunov function 
\begin{equation}
V_{2}=\sum_{i=1}^{m}X_{i}^{T}P_{pi}X_{i}+\sum_{i=1}^{m}\tilde{b}_{i}^{T}\tilde{b}_{i}+\widetilde{\eta}^{T}\Gamma_{p}^{-1}\widetilde{\eta},\label{Lyap1-1}
\end{equation}
the time derivative of (\ref{Lyap1-1}) in view of (\ref{Eq. err passive n order model-2-1})
is given by 
\[
\begin{array}[t]{ccl}
\dot{V}_{2} & = & \sum_{i=1}^{m}\left(X_{i}^{T}\left(A_{pi}^{T}P_{pi}+P_{pi}A_{pi}\right)X_{i}\right),\end{array}
\]
since $A_{pi},\, i=1,\ldots,m$ is Hurwitz, then the Lyapunov equation
(\ref{eq:Lyapunov equation Ppi}) holds. Therefore, one can obtain
\begin{equation}
\dot{V}_{2}=-\sum_{i=1}^{m}X_{i}^{T}Q_{pi}X_{i}\leqslant0.\label{eq:New Lyap func-1}
\end{equation}
Therefore, $X_{i}$, $\tilde{b}_{i}$ and $\widetilde{\eta}_{i}$
are bounded and consequently from (\ref{Eq. err passive n order model-2-1})
and Assumption \ref{assump:boundness omega omega dot} in subsection
\ref{sub:Attitude-kinematics,-Dynamics, assumptions}, $\dot{X}_{i}$,
$\dot{\tilde{b}}_{i}$ and $\dot{\widetilde{\eta}}_{i}$ are also
bounded. The rest of the proof is similar to the proof of Proposition
\ref{prop: High order Direct proposition}. It is easy to verify that
$\ddot{V}_{2}$ is bounded. Thus using Barbalat's lemma, ${\displaystyle \lim_{t\rightarrow\infty}\dot{V}_{2}(t)=0}$
and consequently${\displaystyle \lim_{t\rightarrow\infty}X_{i}(t)=0}$.
In addition, $\ddot{X}_{i}$ are bounded, then ${\displaystyle \lim_{t\rightarrow\infty}\dot{X}_{i}(t)=0}$
and using (\ref{Eq. err passive n order model-2-1}), ${\displaystyle \lim_{t\rightarrow\infty}\widetilde{b}_{i}(t)=0}$.
By a standard reasoning by contradiction, one gets that ${\displaystyle \lim_{t\rightarrow\infty}\dot{\widetilde{b}}_{i}(t)=0}$.
Using this fact and (\ref{Eq. err passive n order model-2-1}), therefore
${\displaystyle \lim_{t\rightarrow\infty}S(b_{i})\widetilde{\eta}=0}$.
Under Assumption \ref{assump:Colinearity of bi}, one can conclude
that ${\displaystyle \lim_{t\rightarrow\infty}\widetilde{\eta}(t)=0}$. \end{proof}
\begin{rem}
Substituting the value of $n$ by 1 in (\ref{Eq. passive n ordre case}),
and after some manipulations, one can obtain the first order passive
filter as

\begin{equation}
\begin{cases}
\begin{array}{lcl}
\dot{\hat{b}}_{i} & = & -S(\omega_{m}-\hat{\eta})\hat{b}_{i}+\gamma_{i1}(b_{i}-\hat{b}_{i}),\\
\dot{\hat{\eta}} & = & \Gamma_{2}\sum_{i=1}^{m}S(b_{i})\hat{b}_{i},
\end{array}\end{cases},\label{Eq. passive n ordre case-1}
\end{equation}

\end{rem}

\section{\label{sec:Attitude Tracking}Attitude tracking using complementary
filter principle}

We propose thereafter a new control law that use only filtered inertial
vectors and rate gyro measurements to track the desired attitude,
without using ``attitude measurements''. The filtered inertial vectors
are obtained using a new filter based on first order direct complementary
filter.

\subsection{Controller Design}

First, let us define the orientation error by $\bar{R}(t)=R(t)R_{d}^{T}(t)$
which corresponds to the quaternion error $\bar{Q}(t)=Q(t)\odot Q_{d}^{-1}(t)\equiv\left[\begin{array}{c}
\bar{q}_{0}(t)\\
\bar{q}(t)
\end{array}\right]\in\mathbb{S}^{3}$ whose dynamics is governed by 
\begin{equation}
\left[\begin{array}{c}
\dot{\bar{q}}_{0}(t)\\
\dot{\bar{q}}(t)
\end{array}\right]=\left[\begin{array}{c}
-\frac{1}{2}\bar{q}^{T}(t)R_{d}(t)\omega(t)\\
\frac{1}{2}\left(\bar{q}_{0}(t)I_{d}+S\left(\bar{q}(t)\right)\right)R_{d}(t)\omega(t)
\end{array}\right],\label{eq:erreur orientation quat}
\end{equation}
where $R_{d}(t)$ is the desired rotation matrix and it's equivalent
unit-quaternion is $Q_{d}(t)$. The angular velocity error is defined
by 
\begin{equation}
\tilde{\omega}(t)=\omega(t)-\omega_{d}(t),\label{eq:Angular velocity tracking errors}
\end{equation}
where $\omega_{d}(t)$ is the desired angular velocity. We now propose
the following new filter designed for the control problem 
\begin{equation}
\dot{\hat{b}}_{i}(t)=-S(\omega)b_{i}+\alpha_{i}(b_{i}(t)-\hat{b}_{i}(t))+S(\omega_{d})(b_{i}(t)-\hat{b}_{i}(t))+\delta_{i}S(b_{i}^{d}(t))\tilde{\omega}(t),\label{eq:new proposed observer for tracking}
\end{equation}
where $\alpha_{i}>0$, $\delta_{i}>0$ ($i=1,\ldots,m$) and the following
new control law 
\begin{equation}
\tau(t)=S(\omega(t))J\omega(t)-JS(\omega_{d}(t))\omega(t)+J\dot{\omega}_{d}(t)+J\sum_{i=1}^{m}\rho_{i}S(b_{i}^{d}(t))\hat{b}_{i}(t)-kJ\tilde{\omega}(t),\label{eq: control law for tracking}
\end{equation}
where $\rho_{i}>0,\, i=1,\ldots,m$, $k>0$ and $\hat{b}_{i}(t)$
is obtained by (\ref{eq:new proposed observer for tracking}).

Using (\ref{dynamic_bi}), (\ref{eq. dynamic de rotation}), (\ref{eq:erreur orientation quat}),
(\ref{eq:new proposed observer for tracking}) and (\ref{eq: control law for tracking})
and define the new variables $\bar{\omega}=R_{d}\tilde{\omega}$ and
$\bar{b}_{i}=R_{d}\tilde{b}_{i}$ one can get the following closed
loop dynamics

\begin{equation}
\begin{cases}
\begin{array}{lcl}
\dot{\bar{b}}_{i}(t) & = & -\alpha_{i}\bar{b}_{i}(t)-\delta_{i}S(r_{i}(t))\bar{\omega}(t),\\
\dot{\bar{q}}_{0}(t) & = & -\frac{1}{2}\bar{q}^{T}(t)\bar{\omega}(t),\\
\dot{\bar{q}}(t) & = & \frac{1}{2}\left(\bar{q}_{0}(t)I_{d}+S\left(\bar{q}(t)\right)\right)\bar{\omega}(t),\\
\dot{\bar{\omega}}(t) & = & -2(\bar{q}_{0}I_{d}-S(\bar{q}))W\bar{q}-\sum_{i=1}^{m}\rho_{i}S(r_{i})\bar{b}_{i}(t)-k\bar{\omega}(t),
\end{array}\end{cases}\label{Eq. closed loop dynamics}
\end{equation}
where $W=-\sum_{i=1}^{m}\rho_{i}S(r_{i})^{2}$, $W$ is a positive
define matrix (see Lemma 1 and Lemma 2 \cite{Tayebi2013}).

Let us define the state $\Theta:=(\bar{b}_{1},...,\,\bar{b}_{m},\,\bar{Q},\,\bar{\omega})$.
The closed loop dynamics (\ref{Eq. closed loop dynamics}) can be
rewritten as $\dot{\Theta}=G(\Theta)$ such that $\Theta\in\Delta$
and $\Delta:=\mathbb{R}^{3m}\times\mathbb{S}^{3}\times\mathbb{R}^{3}$,
and define the following positive radially unbounded function : $V_{3}:\,\varDelta\rightarrow\mathbb{R}$
\begin{equation}
V_{3}(\Theta)=\sum_{i=1}^{m}\frac{\rho_{i}}{\delta_{i}}\bar{b}_{i}^{T}(t)\bar{b}_{i}(t)+2\bar{q}(t)^{T}W\bar{q}(t)+\bar{\omega}(t)^{T}\bar{\omega}(t).\label{eq:lyap func}
\end{equation}

\begin{thm}
\label{thm: attitude tracking}Consider System (\ref{eq:real kinematics})-(\ref{eq. dynamic de rotation})
and the control law (\ref{eq: control law for tracking}) with the
observer given by (\ref{eq:new proposed observer for tracking}).
Under Assumption \ref{assump:Colinearity of bi} in Subsection \ref{sub:Attitude-kinematics,-Dynamics, assumptions}
and if Hypothesis of Lemma 1 in \cite{Benziane2015} holds, then

$(1)$ The equilibria of the closed-loop system (\ref{Eq. closed loop dynamics})
are defined by 
\[
\Theta_{1}^{\pm}=(\underset{m}{\underbrace{\mathbf{0}_{3},...,\mathbf{0}_{3}}},\,\left[\begin{array}{c}
\pm1\\
\boldsymbol{0}
\end{array}\right],\,\boldsymbol{0}),\ \Theta_{2,3,4}^{\pm}=(\underset{m}{\underbrace{\mathbf{0}_{3},...,\mathbf{0}_{3}}},\,\left[\begin{array}{c}
0\\
\pm v_{j}
\end{array}\right],\,\boldsymbol{0}),
\]
where $v_{j}\ ,j=1,2,3$ are the eigenvectors of $W$.

$(2)$ The equilibria $\Theta_{1}^{\pm}$ are asymptotically stable
with a domain of attraction containing the set 
\[
C_{a}^{+}:=\left\{ \Theta\in\triangle\mid V_{3}(\Theta)<4\lambda_{min}(W)\;\hbox{and}\;\bar{q}_{0}>0\right\} ,
\]
for $\Theta_{1}^{+}$ and 
\[
C_{a}^{-}:=\left\{ \Theta\in\triangle\mid V_{3}(\Theta)<4\lambda_{min}(W)\;\hbox{and}\;\bar{q}_{0}<0\right\} ,
\]
for $\Theta_{1}^{-}$, where $\lambda_{min}(W)$ is the smallest eigenvalue
of $W$.

$(3)$ The equilibria $\Theta_{2,3,4}^{\pm}$ are locally unstable
and \textup{$\Theta_{1}^{\pm}$} are almost globally asymptotically
stable. \end{thm}
\begin{proof}
The proof of the first item is similar to the proof of Theorem \ref{thm: attitude tracking}
presented in \cite{Benziane2015}. Recall that the closed loop dynamics
(\ref{Eq. closed loop dynamics}) is autonomous, therefore it is possible
to use LaSalle's invariance theorem to proof the second item. Note
that the time derivative of (\ref{eq:lyap func}) using (\ref{Eq. closed loop dynamics})
is given by $\dot{V}_{3}(\Theta)=-k\bar{\omega}(t)^{T}\bar{\omega}(t)-\sum_{i=1}^{m}\alpha_{i}\frac{\rho_{i}}{\delta_{i}}\bar{b}_{i}(t)^{T}\bar{b}_{i}(t)\leq0$
and the proof of item (2) will be similar to the proof of Theorem
1 presented in \cite{Benziane2015}.

(3) Let us proof that the equilibria $\Theta_{2,3,4}^{\pm}$ are unstable.
Since the only difference between these equilibria is the value of
the eigenvector, the proof is given only for $\Theta_{2}^{+}\in\triangle$
. The other cases will be similar. To do this, we consider $\Theta_{2}^{*}:=(\bar{b}_{1}^{*},...,\,\bar{b}_{m}^{*},\,\bar{Q}^{*},\,\bar{\omega}^{*})$
a neighborhood of $\Theta_{2}^{+}$ (arbitrary close) and since the
function $V_{3}$ is non-increasing, it suffices to prove that $V_{3}(\Theta_{2}^{*})-V_{3}(\Theta_{2}^{+})<0$.
Let us use the following change of variable

\begin{equation}
\bar{Q}^{*}=\left[\begin{array}{c}
\bar{q}_{0}^{*}\\
\bar{q}^{*}
\end{array}\right]=\left[\begin{array}{c}
0\\
v_{1}
\end{array}\right]\odot\left[\begin{array}{c}
x_{0}\\
x
\end{array}\right]=\left[\begin{array}{c}
-v_{1}^{T}x\\
x_{0}v_{1}+S(v_{1})x
\end{array}\right]\label{eq:change_var_inv}
\end{equation}

Using (\ref{eq:change_var_inv}) and the fact that $Wv_{1}=\lambda_{1}v_{1}$
(where $\lambda_{1}$ is the eigenvalue associated to the unit eigenvector
$v_{1}$ of $W$), one can evaluate $D=V_{3}(\Theta_{2}^{*})-V_{3}(\Theta_{2}^{+})$
as follow

\begin{equation}
D=\sum_{i=1}^{m}\frac{\rho_{i}}{\delta_{i}}\bar{b}_{i}^{*T}\bar{b}_{i}^{*}+\bar{\omega}^{*}{}^{T}\bar{\omega}^{*}+4\lambda(x_{0}^{2}-1)-4x^{T}S(v_{1})WS(v_{1})x,\label{eq:difference1_unstable}
\end{equation}

If we take $x$ close to $v_{2}$ such that $x=\varepsilon v_{2}$,
where $\varepsilon>0$ sufficiently small, the unit quaternion constraint
gives $x_{0}^{2}=1-\varepsilon^{2}$. In this case, one can gets $D=\sum_{i=1}^{m}\frac{\rho_{i}}{\delta_{i}}\bar{b}_{i}^{*T}\bar{b}_{i}^{*}+\bar{\omega}^{*}{}^{T}\bar{\omega}^{*}-4\lambda_{1}\varepsilon^{2}$
which means that if $\varepsilon^{2}>\frac{1}{4\lambda_{1}}\left(\sum_{i=1}^{m}\frac{\rho_{i}}{\delta_{i}}\bar{b}_{i}^{*T}\bar{b}_{i}^{*}+\bar{\omega}^{*}{}^{T}\bar{\omega}^{*}\right)$
then $D<0$. As a result, there exist $\Theta_{2}^{*}$ arbitrary
close to $\Theta_{2}^{+}$ such that $V_{3}(\Theta_{2}^{*})<V_{3}(\Theta_{2}^{+})$
and since the function $V_{3}$ is non increasing, it is clear that
$\Theta_{2}^{+}$ is unstable. Similarly, all equilibria $\Theta_{2,3,4}^{\pm}$
are unstable. Finally, in the state space $\triangle$ the set of
unstable equilibria is Lebesgue measure zero. Therefore, almost all
trajectories converge asymptotically to $\Theta_{1}^{\pm}$.\end{proof}
\begin{rem}
In the case of stabilization ($\omega_{d}=0$), the control law (\ref{eq: control law for tracking})
with the filter (\ref{eq:new proposed observer for tracking}) can
be modified to get 
\begin{eqnarray}
\dot{\hat{b}}_{si}(t) & = & \alpha_{i}(b_{i}(t)-\hat{b}_{si}(t))-S(\omega(t))b_{i}(t)+\delta_{i}S(b_{i}^{d})\omega(t),\label{eq:filter in the case of stabilization}\\
\tau_{s}(t) & = & \sum_{i=1}^{m}\rho_{i}S(b_{i}^{d}(t))\hat{b}_{si}(t)-k\omega(t).\label{eq:control in the case of stabilization}
\end{eqnarray}

\end{rem}

\section{\label{sec:Experimentals}Experimental results}

In this section, we present some experimental results showing the
effectiveness and the performances of the proposed solutions. Experiments
were done based on DIY drone project \cite{3DRa}. We have used the
platform shown in Figure \ref{fig:Testbench-DIY-Quad}. It is a test-bench
with DIY Quad equipped with the APM2.6 \cite{3DR} autopilot used
for indoor tests. The autopilot APM2.6 is based on Atmel ATMEGA2560-16AU
using an external clock of 16MHz. The embedded system is equipped
with Invensense's 6 DoF Accelerometer/Gyro MPU-6000 and a 3-axis external
compass HMC5883L-TR. The main loop operating frequency of the firmware
is 100Hz. The acquisition of accelerometer and gyros measurements
is similar to the main loop while the frequency acquisition of magnetometer
measurements is 10 Hz (after an internal filtering).

For experiments, $r_{1}=[0,0,1]^{T}$ and $r_{2}=[0.434,-0.04,0.899]^{T}$
are the gravitational earth vector and magnetic earth filed vector,
respectively, expressed in North East Down ``NED'' reference frame
and both normalized. To validate our results, two main experiments
were done. The first one was made to evaluate the performance of our
attitude observer using the well known Xsens MTi AHRS, as illustrated
in Figure \ref{fig:The-Inertial-Measurements}. In this experiment,
the attitude measurements provided by the MTi is considered as a reference
signal. The second experiment consists of the implementation of our
attitude controller directly on the autopilot APM2.6.

\begin{center}
\begin{figure}
\begin{centering}
\includegraphics[width=9cm]{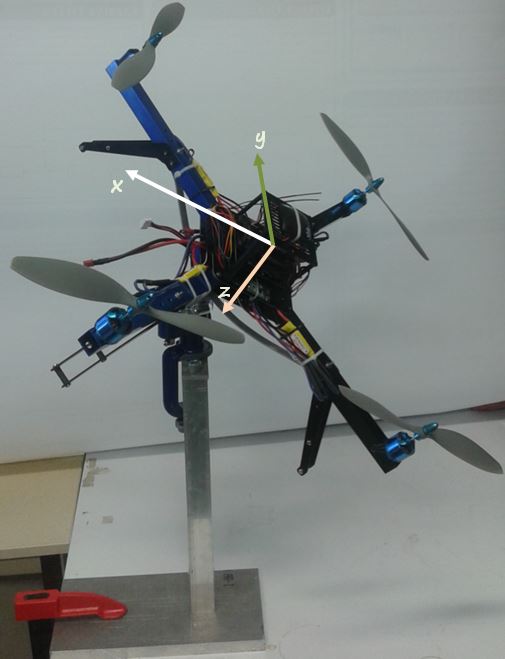}
\par\end{centering}

\protect\caption{\label{fig:Testbench-DIY-Quad}Test-bench DIY Quad}
\end{figure}

\par\end{center}

\subsection{Attitude estimation}

As described above, the attitude measurements delivered by the Xsens
MTi will be considered as a reference signal for the comparison of
results. This reference is obtained with an internal Kalman filter
implemented inside MTi. The explicit observer presented in \cite{Mahony2008}
with quaternion formulation was implemented and will be termed as
\textit{``MHP''} observer. 
\begin{rem}
For simplicity and implementation consideration, only the first order
\textit{``Direct''} and \textit{``Passive''} filters given by
(\ref{Eq. direct n ordre case-1}) and (\ref{Eq. passive n ordre case-1})
were implemented using first order Euler integration, where we take
$i=1,2$, $b_{1}=a=[\begin{array}{ccc}
a_{x} & a_{y} & a_{z}\end{array}]^{T}\,(m/s^{2})$ for accelerometer measurements and $b_{2}=m=[\begin{array}{ccc}
m_{x} & m_{y} & m_{z}\end{array}]^{T}\,(normalized)$ for magnetometer measurements. 
\end{rem}
For implementation, the following gains were chosen: $\gamma_{11}=\gamma_{21}=1$
and $\Gamma_{1}=\Gamma_{2}=0.003I_{d}$ for both two filters while
for \textit{``MHP\textquotedblright{}} observer, the gains presented
in \cite{Mahony2008} were used : $k_{P}=1$ and $k_{I}=0.3$. The
measured initial attitude condition given by MTi was $Q(0)=[0.998,\:-0.031,\:-0.029,\:-0.046]^{T}$,
which was used as initial condition for \textit{``MHP\textquotedblright{}}
observer and the equivalent initial conditions for \textit{``Direct''}
and \textit{``Passive''} proposed filters were $a(0)=[0.771,\:-0.796,\:9.652]^{T}$
and $m(0)=[0.049,\:0.016,\:-0.263]^{T}$. For reporting results, we
first consider the performance of the data fusion obtained by implemented
complementary filters. Then, figures \ref{fig:Complementary-filtering-experime acc}
and \ref{fig:Complementary-filtering-experime-mag} show experimental
results for the direct and passive filters. One can observe that the
two complementary filters have similar performance which corroborates
the fact that asymptotic stability were demonstrated for both filters.
As explained before, the passive filter is less sensitive to noise.
This can be illustrated in Figure \ref{fig:Complementary-filtering-experime acc}-(c).
Note that the raw magnetometer measurements are not very corrupted
by noise as illustrated in Figure \ref{fig:Complementary-filtering-experime-mag}
and this due to the fact that they were already filtered inside the
MTi. Thereafter, the outputs of theses filters are used to estimate
attitude using TRIAD algorithm as illustrated in Figure \ref{fig:Attitude-estimation-experimental}.
In this figure, the estimated attitude is compared to that obtained
with the raw measurements. The comparison presented in Figure \ref{fig:Attitude-estimation-experimental-comp}
illustrate the effectiveness of the proposed observer compared to
Kalman filter (implemented inside MTi) or \textit{``MHP\textquotedblright{}}
observer. In Figure \ref{fig:Rate-gyro-bias}, the gyros bias estimation
from both observers is shown and both two observers give roughly similar
results.

\begin{center}
\begin{figure}
\begin{centering}
\includegraphics[width=9cm]{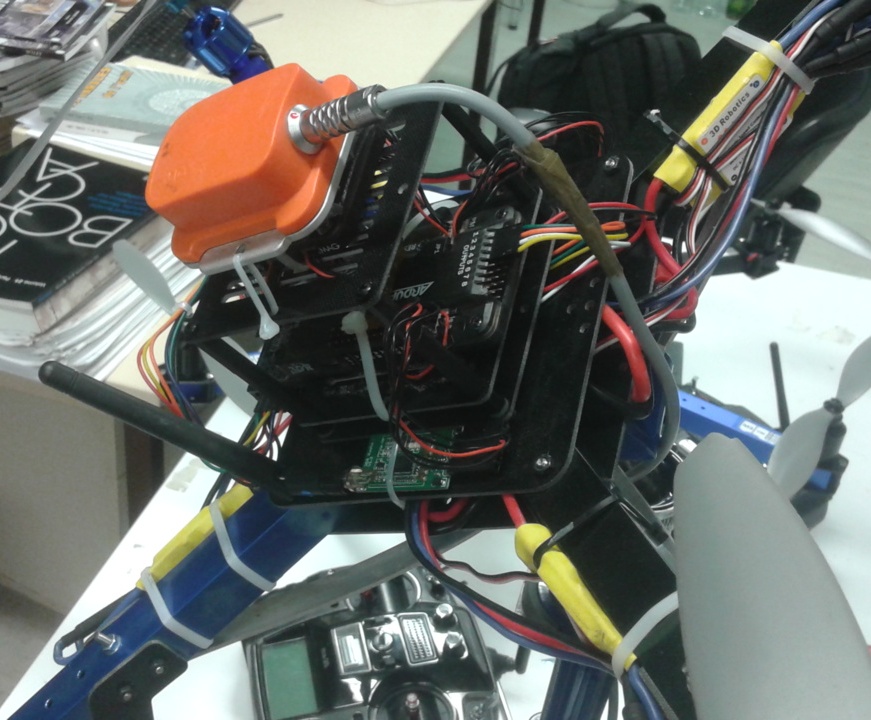}
\par\end{centering}

\protect\caption{\label{fig:The-Inertial-Measurements}The Inertial Measurements Unit
Xsens mounted on the test-bench}
\end{figure}

\par\end{center}

\begin{center}
\begin{figure}
\begin{centering}
\includegraphics[width=9cm]{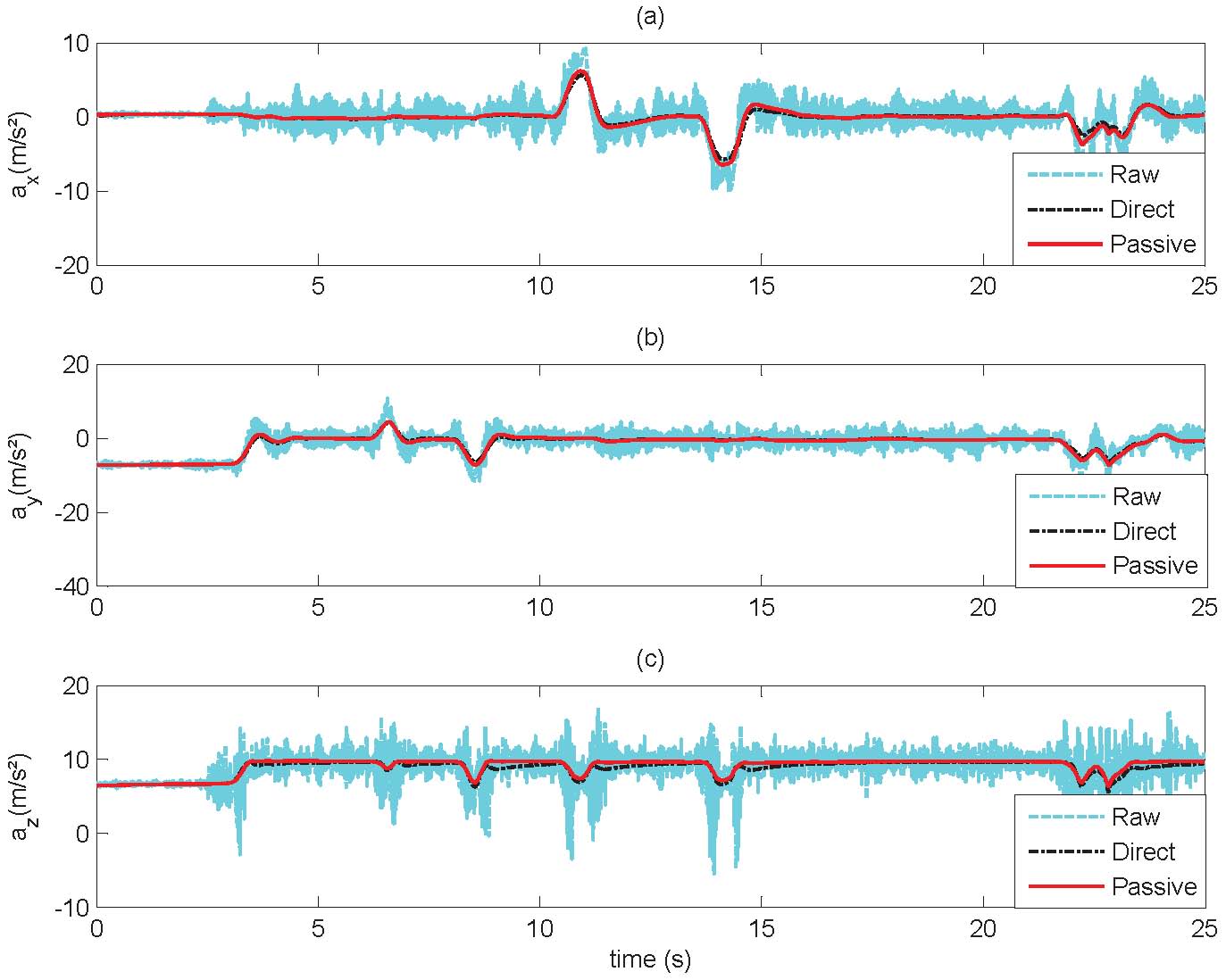} 
\par\end{centering}

\protect\caption{\label{fig:Complementary-filtering-experime acc}Complementary Accelerometer
filters experimental results}
\end{figure}

\par\end{center}

\begin{center}
\begin{figure}
\begin{centering}
\includegraphics[width=9cm]{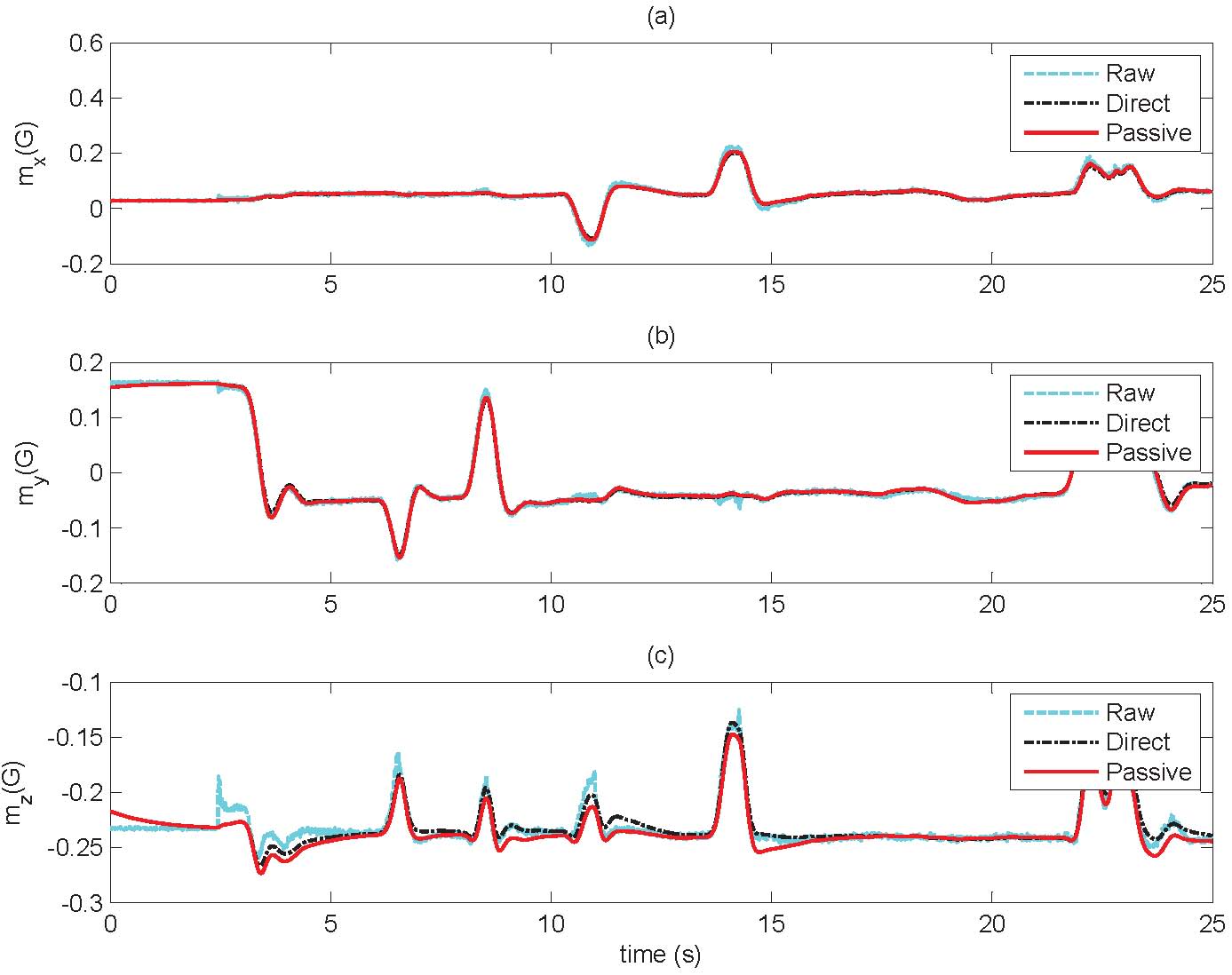} 
\par\end{centering}

\protect\caption{\label{fig:Complementary-filtering-experime-mag}Complementary Magnetometer
filters experimental results}
\end{figure}

\par\end{center}

\begin{center}
\begin{figure}
\begin{centering}
\includegraphics[width=9cm]{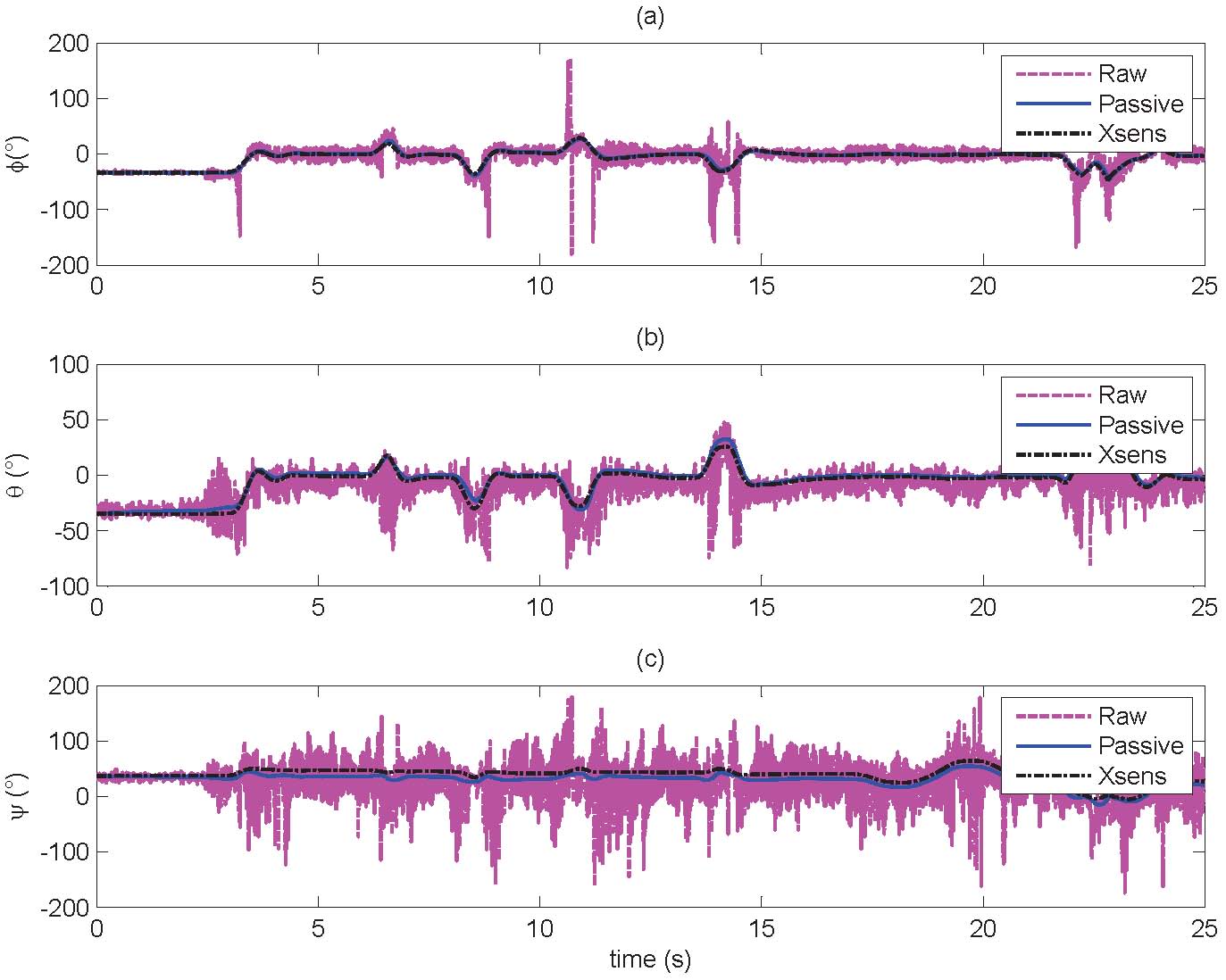} 
\par\end{centering}

\protect\caption{\label{fig:Attitude-estimation-experimental}Attitude estimation experimental
results for the proposed observers}
\end{figure}

\par\end{center}

\begin{center}
\begin{figure}
\begin{centering}
\includegraphics[width=9cm]{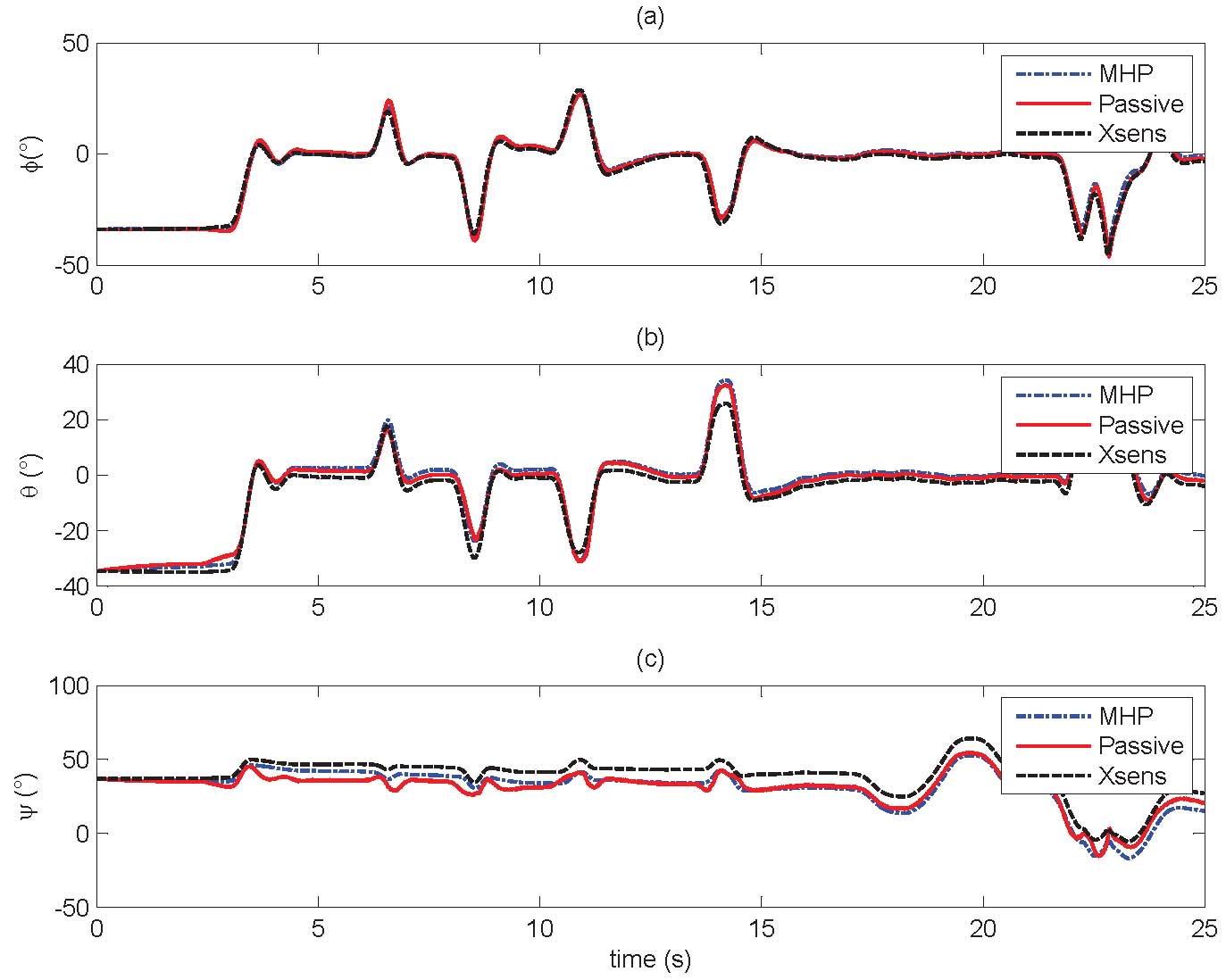} 
\par\end{centering}

\protect\caption{\label{fig:Attitude-estimation-experimental-comp}Attitude estimation
experimental results comparison}
\end{figure}

\par\end{center}

\begin{center}
\begin{figure}
\begin{centering}
\includegraphics[width=9cm]{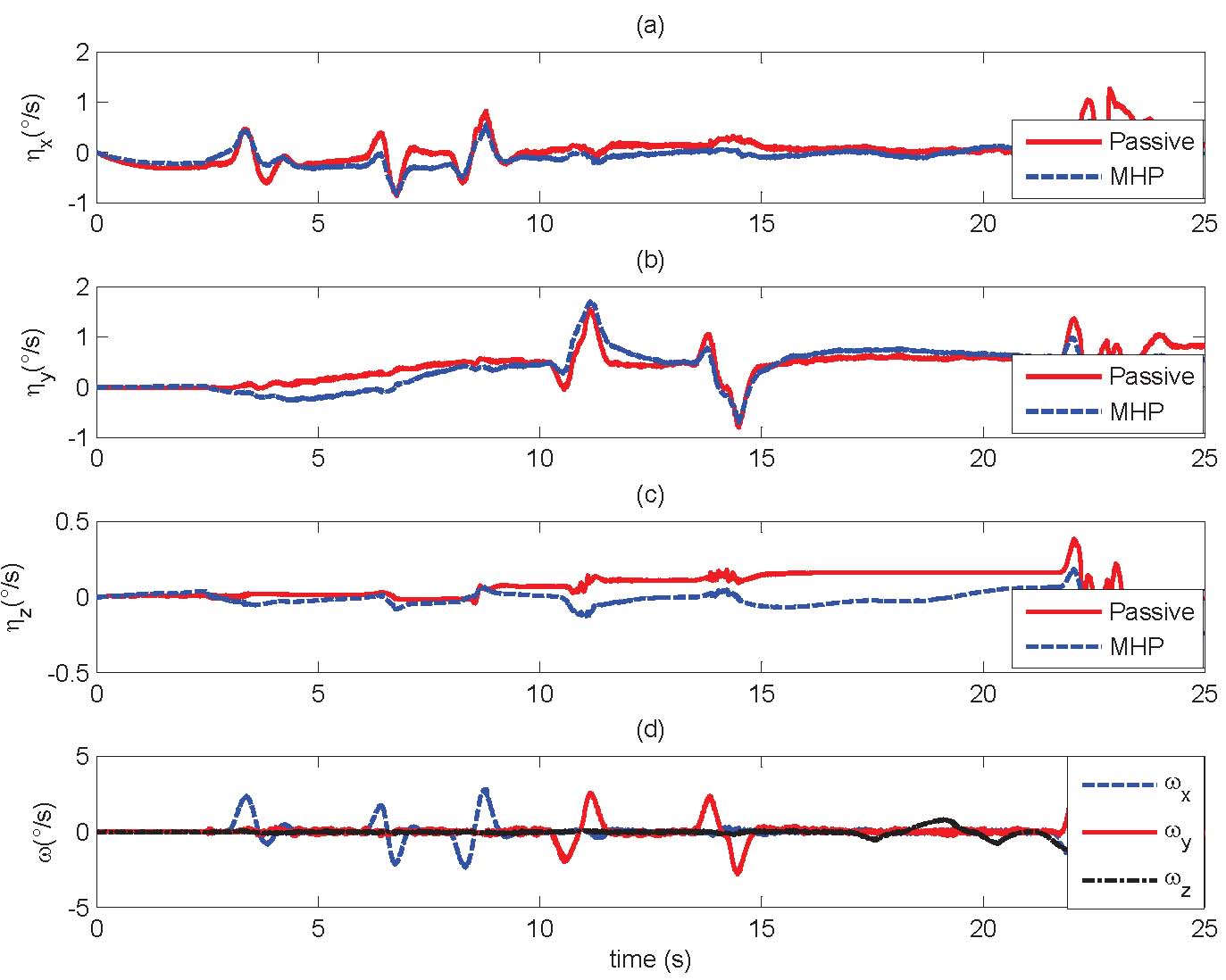} 
\par\end{centering}

\protect\caption{\label{fig:Rate-gyro-bias}Rate gyro bias estimation experimental
results}
\end{figure}

\par\end{center}

\subsection{Attitude stabilization}

For this test, we considered for simplicity and without loss of generality
the special case of stabilization of attitude. The experiment was
done using the test-bench shown in Figure \ref{fig:Testbench-DIY-Quad}.
The controller (\ref{eq:control in the case of stabilization}) was
implemented using the following notations and parameters : $R_{d}(t)=I$,
which means $b_{1}^{d}=r_{1}$ and $b_{2}^{d}=r_{2}$; $\hat{b}_{1}=\hat{a}\:(normalized)$,
$\hat{b}_{2}=\hat{m}\:(normalized)$ are the estimates of the inertial
vector measurements given by the accelerometer and magnetometer, respectively;
$\omega(t)\:(rad/s)$ is the rate gyro measurements; $\rho_{1}=1.66$
and $\rho_{2}=0.1161$ (for the axis $x$ and $y$), and $\rho_{1z}=0.05$
and $\rho_{2z}=0.03$ (for the $z$ axis); The damping $k=0.2621$
and the filter gains $\alpha_{1}=6$ and $\alpha_{2}=10$.

The main loop for attitude stabilization is running at 100Hz. At each
loop the measurements of accelerometer and magnetometer are normalized
after the execution of the observer (\ref{eq:new proposed observer for tracking}).
Due to the poor quality of magnetometer measurements the gains corresponding
to $z$ axis are chosen small. Therefore, the stabilization is done
around $x$ and $y$ axis only. Then, starting from an arbitrary measured
initial condition in Euler angles $(\phi,\theta,\psi)=(-18.478,41.192,2.847)\text{\textdegree}$,
the evolution of normalized inertial measurements vectors, torque
and Euler angles are shown in Figure \ref{fig:Experimental-results}.
We can see that after transient time, the normalized measurements
vectors $a$ and $m$ converge to the desired values $b_{1}^{d}=[0,0,1]^{T}$
and $b_{2}^{d}=[0.434,-0.04,0.899]^{T}$. Consequently according with
the attitude estimate, this corresponds to the roll and pitch angles
close to zero which confirms the stabilization of the platform. We
can also observe that control torque is smooth without noise through
the use of the complementary filter.

\begin{center}
\begin{figure}
\begin{centering}
\includegraphics[width=9cm]{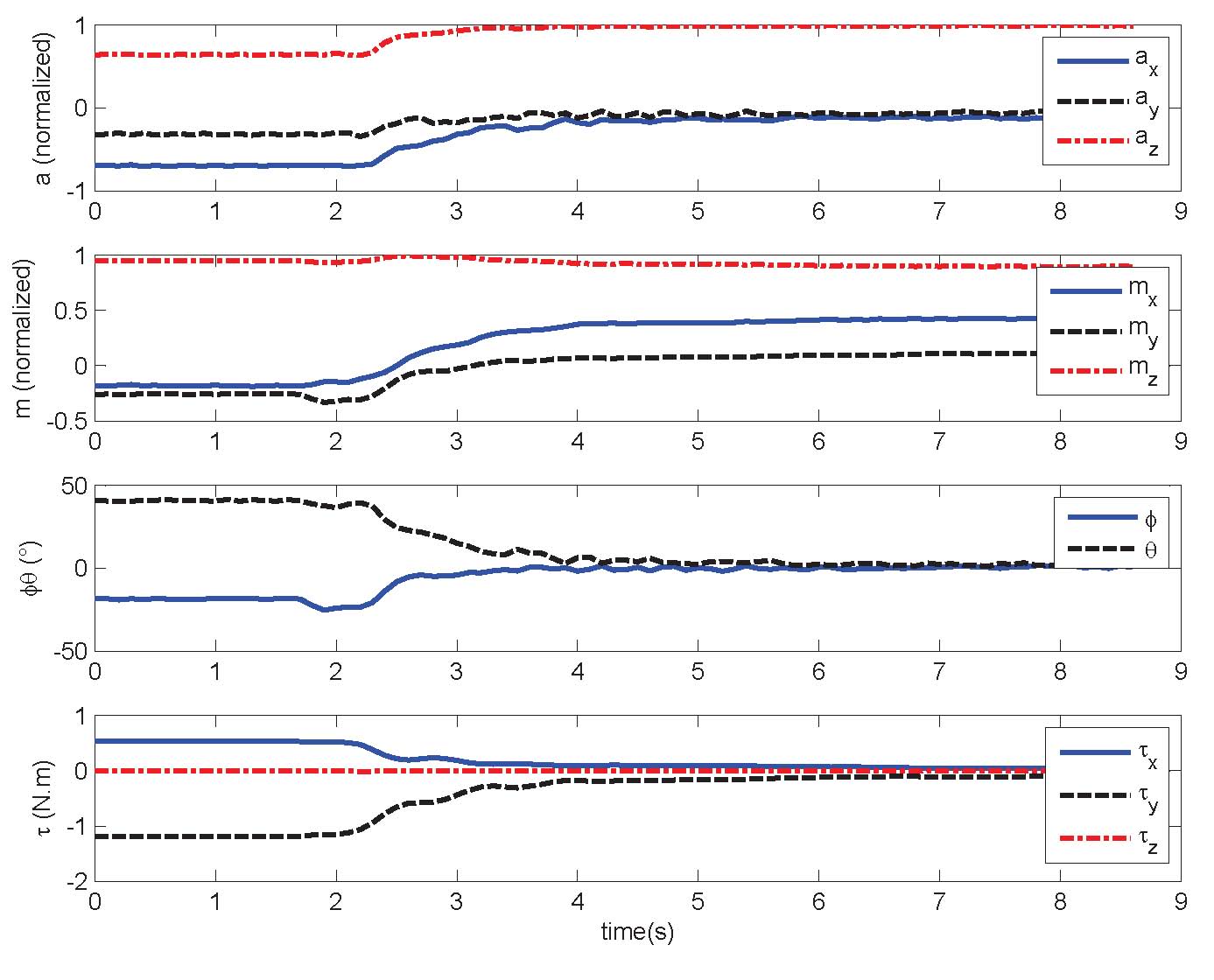} 
\par\end{centering}

\protect\caption{\label{fig:Experimental-results}Attitude stabilization experimental
results}
\end{figure}

\par\end{center}

\section{\label{sec:conclusion}Conclusions}

Due to its importance and despite the considerable number of solutions,
the problem of attitude estimation and control is still relevant.
This paper presents High order \textit{``Direct''} and \textit{``Passive''}
linear-like complementary filters for attitude and gyro-rate bias
estimation. Using Lyapunov analysis, the proposed solutions ensure
global convergence. Another novelty of this work lies in the proposition
of new control law for attitude tracking problem, in which the principle
of data fusion is used. Only filtered inertial vectors and rate gyro
measurements were used in the control law, without using ``attitude
measurements'' and ensuring an almost global stability. To show the
efficiency and performance of the proposed solutions, a set of experimental
tests were performed based on DIY drone Quadcopter, equipped with
APM2.6 autopilot. The passive second order filter can be of great
help. Indeed, in future work, this filter will be used to enhance
the low sampling frequency of magnetometer measurements compared to
that of accelerometer.

\appendices{}

\appendix{Proof of Lemma 1}

Showing the thesis amounts to exhibit an example. For that purpose,
consider $\gamma=(C_{n}^{l}\alpha^{l})_{1\leq l\leq n}\in\mathbb{R}^{n}$,
where $n$ is a positive integer, $\alpha$ a positive real number
and the $C_{n}^{l}$ are the binomial coefficients. Then $P_{\gamma}(s)=(s+\alpha)^{n}$
implying that $\gamma\in\mathcal{H}_{n}$. It remains to show that
$\gamma\in\mathcal{\overline{H}}_{n}$. One clearly has that $P_{\pi(\gamma)}=(P_{\gamma}(s)-P_{\gamma}(0))/s$
and thus the roots of $P_{\pi(\gamma)}$ are the non zero roots of
$(s+\alpha)^{n}-\alpha^{n}$. Every root $z$ of the previous polynomial
verifies that $(\frac{z}{\alpha}+1)^{n}=1$ and then $\frac{z}{\alpha}+1=e^{j\frac{2k\pi}{n}}$,
where $j^{2}=-1$ and $k=0,\ldots,n-1$. It yields that $Re(z)=\alpha(\cos(\frac{2k\pi}{n})-1)$,
which is negative only if $k\neq0$ and in the latter case $z=0$.
One deduces that all the roots of $P_{\pi(\gamma)}$ have negative
real part, i.e., $P_{\pi(\gamma)}$ is Hurwitz and thus $\gamma\in\mathcal{\overline{H}}_{n}$.

\bibliographystyle{IEEEtran}
\bibliography{Biblio_IJC}

\end{document}